\documentclass[reqno,11pt]{amsart}

\usepackage[margin=1in, letterpaper, portrait]{geometry}

\usepackage{
amssymb,
amsmath,
amsthm,
amsbsy,
graphicx,
tabularx,
mathtools,
float, 
enumerate,
ytableau,
tikz
}

\usepackage[vcentermath,enableskew]{youngtab}

\usepackage[utf8]{inputenc}

\usepackage[colorlinks=true, pdfstartview=FitV, 
linkcolor=violet, 
citecolor=blue,
filecolor=violet, urlcolor=violet]{hyperref}

\usepackage{orcidlink}

\theoremstyle{definition} % make everything non-italics
\newtheorem{theorem}{Theorem}[section]
\newtheorem{lemma}[theorem]{Lemma}
\newtheorem{proposition}[theorem]{Proposition}
\newtheorem{corollary}[theorem]{Corollary}
\newtheorem{definition}[theorem]{Definition}
\newtheorem{example}[theorem]{Example}

\numberwithin{equation}{section}
\counterwithout{figure}{section}
\newtheorem*{theorem*}{Theorem}

\DeclareMathOperator{\sh}{sh}
\renewcommand{\P}{\text{P}}
\DeclareMathOperator{\Q}{Q}

\newcommand{\run}{\textnormal{\textsf{run}}} % function counting minimal number of runs among all reduced words of a permutation

\newcommand{\rw}[1]{\left[{#1}\right]} % notation for reduced words

\newcommand{\supp}{\textnormal{\textsf{supp}}} % support of a permutation

\newcommand{\canonicalword}[1]{\textnormal{\textsf{canon}}(#1)} % canonical reduced word

\newcommand{\Rowone}{{\sf{{Row}}}_1}

\newcommand{\Rowtwo}[1]
{{\sf{{Row}}}_2(\P(#1))}

\newcommand{\RowTwo}
{{\sf{Row}}_2}

\newcommand{\RowtwoQ}[1]
{{\sf{{Row}}}_2(\Q(#1))}

\DeclareMathOperator{\Incr}{\textcolor{red}{Inc}}

\DeclareMathOperator{\Decr}{\textcolor{red}{Dec}}

\title{Runs and RSK tableaux of boolean permutations}

\author[Gunawan]{Emily Gunawan$^*$ \orcidlink{0000-0001-5056-0792}} 
\address{Department of Mathematics and Statistics, University of Massachusetts Lowell, Lowell, MA, USA}
\email{emily$\_$gunawan@uml.edu}
\thanks{$^*$Research partially completed at the Isaac Newton Institute for Mathematical Sciences  during the programme Cluster algebras and representation theory (supported by EPSRC Grant Number EP/R014604/1).}

\author[Pan]{Jianping Pan \orcidlink{0000-0001-6722-8091}} 
\address{Department of Mathematics, North Carolina State University, Raleigh, NC, USA}
\email{jpan9@ncsu.edu}

\author[Russell]{Heather M. Russell \orcidlink{0000-0003-1256-2150}} 
\address{Department of Mathematics and Statistics, University of Richmond, Richmond, VA, USA}
\email{hrussell@richmond.edu}

\author[Tenner]{Bridget Eileen Tenner$^\dagger$ \orcidlink{0000-0003-0150-9653}}
\address{Department of Mathematical Sciences, DePaul University, Chicago, IL, USA}
\email{bridget@math.depaul.edu}
\thanks{$^\dagger$Research partially supported by NSF Grant DMS-2054436 and Simons Foundation Collaboration Grant for Mathematicians 277603.}


\begin{document}

\begin{abstract}
We define and construct the ``canonical reduced word'' of a boolean permutation, and show that the RSK tableaux for that permutation can be read off directly from this reduced word. We also describe those tableaux that can correspond to boolean permutations, and enumerate them. In addition, we generalize a result of Mazorchuk and Tenner, showing that the ``run'' statistic influences the shape of the RSK tableau of arbitrary permutations, not just of those that are boolean.\\

\end{abstract}

\maketitle

\section{Introduction}

In 1961, Schensted showed that the first part of an RSK partition
$(\lambda_1(w),\lambda_2(w),\dots)$ 
is equal to the length of a longest increasing subsequence in a permutation~\cite{Sch61}. In 1974, Greene generalized Schensted’s result, showing that the RSK partition of a permutation records the numbers of disjoint unions of increasing sequences of the permutation~\cite{Gre74}. As he pointed out in his paper, this result was ``somewhat surprising,'' since there was no concrete interpretation of each individual part of the RSK partition below the first part.

The boolean elements in a Coxeter group are those elements whose principal order ideals in the Bruhat order are isomorphic to boolean algebras. This is an important class of elements, established in \cite{tenner-patt-bru}, which has beautiful topological (\cite{DEP21,Kar13, ragnarsson-tenner-homotopy, ragnarsson-tenner-homology}) and representation theoretic (\cite{MT22}) properties. 
Previous work established that the RSK shape of a boolean permutation has at most two rows. 
Recently, Mazorchuk and Tenner \cite[Theorem~6.4]{MT22} showed that the size of the RSK partition of a boolean permutation excluding the first row equals its run statistic, a statistic on the reduced words of a permutation.  
In Theorem~\ref{thm:run and lambda1} of the present paper, we generalize this result to all permutations, as follows. 
\begin{theorem*} 
For \emph{any} permutation $w \in S_n$, we have
$$\lambda_1(w) + \run(w) = n.$$
\end{theorem*}

Our result gives a concrete interpretation to the sum of all individual parts below the first row; in the case of fully commutative permutations, this gives meaning to the length of the second row and, by doing so, takes a step to address the missing meaning mentioned in~\cite{Gre74}.

Theorem~\ref{thm:run and lambda1} demonstrates a connection between the Coxeter-perspective and the pattern-\break perspective of permutations, via the RSK correspondence. 
From there, we focus exclusively on boolean permutations, defining a ``canonical reduced word,'' which we can construct in two (equivalent) ways. 
Not only does this word demonstrate the necessary run statistic, but we show, in fact, that it directly determines the entire RSK tableaux of the boolean permutation --- without using the insertion algorithm.

The relationship between this canonical reduced word and the tableaux can be exploited further, allowing us to characterize precisely which tableaux are the RSK insertion and recording tableaux of boolean permutations. In particular, only those tableaux that we call ``uncrowded'' can correspond to these permutations under RSK. The uncrowded tableaux, themselves, have interesting combinatorics, as we demonstrate via their enumeration.

The paper is organized as follows. In Section~\ref{sec:background}, we introduce notation and terminology that we will use throughout the paper. We also review heap posets of boolean permutations, RSK insertion, and runs. In Section~\ref{sec:runs.and.rs.partitions}, we give a concrete interpretation of the run statistic in terms of the shape of arbitrary permutations in Theorem~\ref{thm:run and lambda1}. 
Section~\ref{sec.second.row} introduces the canonical reduced word of a boolean permutation, which can be defined (and constructed) from either an arbitrary reduced word or from its heap. 
The canonical word is used to construct the corresponding insertion and recording tableaux in Theorem~\ref{thm:boolean second row}, and it demonstrates the run statistic (Corollary~\ref{cor:canonicalword is optimal run word}). 
In Section~\ref{sec:Characterizing boolean insertion tableaux}, we characterize tableaux that are RSK tableaux of boolean permutations, which we call ``uncrowded'' tableaux. This is done in  Corollaries~\ref{cor:second row rules for balanced tableaux} and~\ref{cor:second row rule for recording tableaux}. We conclude with Section~\ref{sec:counting uncrowded tableaux}, enumerating the uncrowded tableaux via a bijection with binary words in which each maximal block of $1$s has odd length.

\section{Background and notation}\label{sec:background}
Let $S_n$ be the symmetric group on $n$ elements. 
We represent permutations of $S_n$ in \emph{one-line notation} as $w = w(1)w(2)\cdots w(n)$. For each $i \in \{1, \ldots, n-1\}$, let $s_i \in S_n$ denote the \emph{simple reflection} (also called an \emph{adjacent transposition}) 
swapping $i$ and $i+1$, and fixing all other letters. The simple reflections generate $S_n$, meaning that every $w \in S_n$
can be decomposed as a product $w = s_{i_1} \cdots s_{i_\ell}.$ 
The minimum $\ell$ among all such decompositions for $w$ 
is the \emph{(Coxeter) length} of $w$, denoted $\ell(w)$. In our proofs, we will make use of the fact that $\ell(w)$ is also the number of \emph{inversions} in the one-line notation for $w$ where an inversion is a pair of positions $i<j$ such that $w(i)>w(j)$.
An expression $w=s_{i_1} \cdots s_{i_{\ell(w)}}$ is called
a \emph{reduced decomposition} of $w$.   We ease this notation by writing such a decomposition as the \emph{reduced word} $\rw{i_1 \cdots i_{\ell(w)}}$. Let $R(w)$ denote the set of reduced words for $w$.

The \emph{support} $\supp(w)$ of a permutation $w$ is the set of letters appearing in reduced words of $w$. Although  simple reflections are subject to the Coxeter relations, this does not change the set of reflections appearing in any reduced decomposition, so $\supp(w)$ is well defined. 
A permutation $w \in S_n$ has \emph{full support} if $\supp(w) = \{1,\ldots, n-1\}$. 
For example, let  $w = 51342 = s_{4}s_{2}s_{3}s_{2}s_{4}s_{1} \in S_5.$ The permutation $w$ has six inversions, so $\ell(w)=6$ and $\rw{423241}\in R(w)$. Since $\supp(w) = \{1,2,3,4\}$, we conclude that $w$ has full support.

\subsection{Fully commutative permutations and boolean permutations}

Let $w\in S_n$ and $\sigma\in S_m$ with $m\leq n$. The permutation $w$ is said to \emph{contain the pattern} $\sigma$ if $w$ has a (not necessarily contiguous) subsequence
whose elements are 
in the same relative order as $\sigma$.  If $w$ does not contain $\sigma$, we say that $w$ \emph{avoids} $\sigma$. For example, 
$w$ = 314592687 contains the pattern 1423 because the subsequence 4968 (among others) is ordered in the same way as 1423. On the other hand, $w$ avoids 3241 since it has no subsequence ordered in the same way as 3241. Note also that each inversion of a permutation is an instance of a 21-pattern.

Simple reflections satisfy \emph{commutation relations} of the form $s_i s_j=s_j s_i$ when $|i-j| > 1$. An application of a commutation relation is called a \emph{commutation move}. 
When referring to reduced words, 
we will say adjacent letters $i$ and $j$ in a reduced word \emph{commute} when $|i - j| > 1$. Given a reduced word $\rw{s}$ of a permutation, the equivalence class consisting of all words that can be obtained from $\rw{s}$ by a sequence of commutation moves is the \emph{commutation class} of $\rw{s}$.
A permutation whose set of reduced words forms a single commutation class is called \emph{fully commutative}. The following proposition characterizes fully commutative permutations in terms of pattern avoidance.

\begin{proposition}[\cite{billey-jockusch-stanley}]
\label{prop:tfae fully commutative}
Let $w$ be a permutation.
The following are equivalent:
\begin{itemize}
\item $w$ is fully commutative,
\item $w$ avoids the pattern 321,
\item no reduced word of $w$ contains $i(i+1)i$ as a factor, for any $i$, and 
\item no reduced word of $w$ contains $(i+1)i(i+1)$ as a factor, for any $i$.
\end{itemize}
\end{proposition}

This paper focuses on the subset of fully commutative permutations known as boolean permutations. 
While boolean permutations can be characterized using the language of the Bruhat order (see, for example, \cite[Chapter 2]{BB05}),
the following result provides a description analogous to that of Proposition~\ref{prop:tfae fully commutative}.   
\begin{proposition}[\cite{tenner-patt-bru}]\label{prop:tfae boolean}
Let $w$ be a permutation.
The following are equivalent:
\begin{itemize}
\item $w$ is boolean,
\item $w$ avoids the pattern 321 and 3412,
\item there is a reduced word of $w$ that consists of all distinct letters, and
\item every reduced word of $w$ consists of all distinct letters.
\end{itemize}
\end{proposition}

\subsection{Heaps of a boolean permutation}
\label{sec:heap of boolean permutation}
Heaps are posets used in~\cite{Ste96} to study fully commutative elements of Coxeter groups. In this paper, heaps for boolean permutations provide a useful visualization of a key construction in Section \ref{subsection:canonical word}. 
For a detailed list of attributions on the theory of heaps, see~\cite[solutions to Exercise 3.123(ab)]{Sta12}.

By Proposition \ref{prop:tfae boolean}, a boolean permutation $w$ has the property that each letter in $\supp(w)$ appears exactly once in every reduced word of $w$. 
Thus the relative positions of every pair of consecutive letters $i$ and $i+1$ 
in reduced words of $w$ are fixed. 
This allows 
one to give the following simple description of the heap of a boolean permutation.

\begin{definition}
\label{defn:heap for boolean permutation}
Given a boolean permutation $w \in S_n$, 
the heap 
$H_w$ of $w$ 
is the partial order 
on $\supp(w)$ obtained via the transitive closure of the cover relations 
\[
\text{
$i \prec i+1$ if $i$ appears to the left of $i+1$ in every reduced word of $w$}
\]
and 
\[
\text{
$i \succ i+1$ otherwise.}
\]
We will refer to the Hasse diagram of a heap as a \emph{heap diagram}.
\end{definition}

 The following proposition explains the connection between the heap of a boolean permutation and its complete set of reduced words.

\begin{proposition}
[{\cite[proof of Proposition~2.2]{Ste96} and~\cite[solutions to Exercise 3.123(ab)]{Sta12}}]
\label{prop:set of labeled linear extensions is commutativity class}
If $w$ is a boolean permutation, then 
the set of linear extensions of the heap $H_{w}$ is the set of reduced words of $w$.
\end{proposition}

We conclude this subsection with an example illustrating the heap of a boolean permutation.

\begin{example}\label{ex:heap314569278}
The permutation $w=314569278\in S_9$ has $\rw{21873456}$ as a reduced word, and the permutation is therefore boolean. The heap diagram for $w$ is depicted in Figure~\ref{fig:fence_reduced_word21873456}. By Proposition \ref{prop:set of labeled linear extensions is commutativity class}, the elements of $R(w)$ are exactly the linear extensions of this heap. For instance, $\rw{87213456}\in R(w)$.

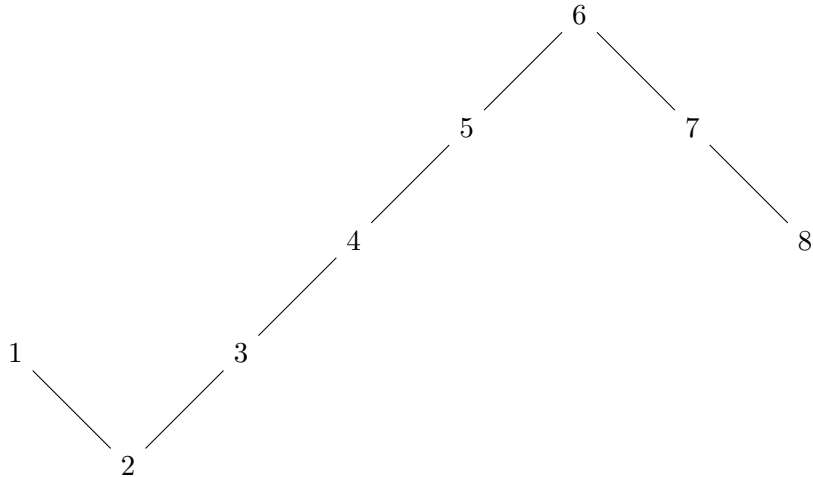
\begin{figure}[htb!]
\centering
\begin{tikzpicture}[xscale=1.5,yscale=1.5,>=latex]
\def\posetedgecolor{blue}
\def\posetedgenegativeslope{red}
\node(1) at (1,1) {$1$}; 
\node(2) at (2,0) {$2$}; 
\node(3) at (3,1) {$3$}; 
\node(4) at (4,2) {$4$};  
\node(5) at (5,3) {$5$};  
\node(6) at (6,4) {$6$}; 
\node(7) at (7,3) {$7$}; 
\node(8) at (8,2) {$8$};

\draw[-] (1) -- (2); 
\draw[-] (2) -- (3); 
\draw[-] (3) -- (4); 
\draw[-] (4) -- (5); 
\draw[-] (5) -- (6); 
\draw[-] (6) -- (7);
\draw[-] (7) -- (8);

\end{tikzpicture}
\caption{The heap diagram for the boolean permutation 
$314569278$.}\label{fig:fence_reduced_word21873456}
\end{figure}
\end{example}

 \subsection{Robinson--Schensted--Knuth tableaux} 
 Given a partition $\lambda=(\lambda_1, \lambda_2, \ldots) \vdash n$, the \emph{Young diagram of shape $\lambda$} is a top- and left-justified collection of $n$ boxes such that the $i^{th}$ row has $\lambda_i$ boxes. A \emph{standard Young tableau of shape $\lambda$} is a filling of the Young diagram of shape $\lambda$ by the values $1,\ldots, n$ such that each value appears exactly once and values increase from left to right in rows and from top to bottom in columns. 
 
The  
Robinson--Schensted--Knuth (RSK) correspondence
as described in~\cite{Sch61}
is 
a bijection 
\[{w \mapsto (\P(w),\Q(w))}\]
from $S_n$ onto pairs of standard Young tableaux of size $n$ having identical shape. The tableau $\P(w)$ is  called the \emph{insertion tableau} of $w$, and the tableau $\Q(w)$ is the \emph{recording tableau} of $w$. The shape of these tableaux is called the \emph{RSK partition of $w$}, denoted $\sh(w) = (\lambda_1(w), \lambda_2(w),\ldots)$. We will also write $\P_i(w)$ to denote the partial insertion tableau constructed by the first $i$ letters, $w(1)\cdots w(i)$, in the one-line notation for $w$. For more details, including a precise description of the RSK insertion algorithm, see, for example~\cite[Section~7.11]{sta99}.  

The following symmetry result is a feature of the RSK insertion algorithm, and one that will simplify our own work.

\begin{proposition}[\cite{Scu63}]\label{prop:P w inverse is Q w}
For any permutation $w$, 
$$\P(w^{-1})=\Q(w).$$
\end{proposition}

Schensted's theorem~\cite[Theorem~1]{Sch61}, stated below, articulates an important relationship between the RSK partition shape and the one-line notation for $w$.

\begin{theorem}\label{thm:Schensted theorem}
Let $w$ be a permutation with RSK partition shape $\sh(w)=(\lambda_1(w), \lambda_2(w), \ldots)$, and let $(\mu_1(w), \mu_2(w),\ldots)$ denote the conjugate of $\sh(w)$. The length of a longest increasing (resp., decreasing) subsequence in the one-line notation of $w$ is $\lambda_1(w)$ (resp., $\mu_1(w)$).
\end{theorem}

By Proposition \ref{prop:tfae fully commutative}, fully commutative permutations are 321-avoiding and therefore have decreasing subsequences of length at most two. It follows from Theorem \ref{thm:Schensted theorem}, then, that the  RSK partitions for fully commutative --- and, in particular, for boolean --- permutations have at most two rows. That observation is key to the arguments in this paper.

Theorem \ref{thm:Schensted theorem} is sufficient for the purposes of this paper,  
but we point out that Greene's theorem~\cite[Theorem ~3.1]{Gre74} is an important generalization of that result, and could perhaps be useful in extensions of our work. For more details about Greene's theorem, see, for example,~\cite[Chapter 3]{Sagan}.

\subsection{Runs and longest increasing subsequences}\label{subsection.run.longest.increasing.subsequences}

In this paper, we define a \emph{run} as an increasing or decreasing sequence of consecutive integers; for example, $234$ and $432$ are runs, but $245$ and $542$ are not. 
Given a permutation $w$, 
let $\run(w)$ denote the fewest number of runs needed to form a reduced word for $w$. 
A reduced word $\rw{s} \in  R(w)$
that can be written as the concatenation of $\run(w)$ runs is called an \emph{optimal run word for $w$}. 

\begin{example}\label{ex:optimal run words}
Consider the permutation $w=345619278 \in S_9$. Examining all reduced words for $w$ shows that $\run(w) = 3$, and thus the reduced words $\rw{21873456}$ and $\rw{87213456}$ given in Example~\ref{ex:heap314569278}  
are both optimal run words for $w$. We can highlight their runs by writing them as $\rw{21\cdot87\cdot3456}$ and $\rw{87\cdot21\cdot3456}$, respectively. In contrast, $\rw{82713456} \in R(w)$ is not optimal because the string $82713456$ cannot be written as the concatenation of three runs.
\end{example}

Recently, Mazorchuk and Tenner~\cite[Theorem 6.4]{MT22} showed that, if $w$ is a boolean permutation, then $\lambda_2(w) = \run(w)$. 
Because the RSK partitions of boolean permutations have at most two rows, we can apply Theorem \ref{thm:Schensted theorem} to conclude that the length of a longest increasing subsequence of a boolean permutation $w \in S_n$  is equal to $n - \run(w)$. In the present work, we will see that this result is not only true for boolean permutations, but for all permutations.

\section{Runs and RSK partitions}\label{sec:runs.and.rs.partitions}

The goal of this section is to prove Theorem~\ref{thm:run and lambda1}, which relates the first row of the RSK partition to the $\run$ statistic, for all permutations. This will generalize \cite[Theorem 6.4]{MT22}, which was a result for boolean permutations, and the argument from that paper will guide this more general setting.

\begin{proposition}[{\cite[Lemma 6.2 and Corollary 6.3]{MT22}}]\label{prop:lower bound on runs}
For any permutation $w \in S_n$, we have $n - \lambda_1(w) \le \run(w)$.
\end{proposition}

\begin{proof}
Although \cite{MT22} states these results in terms of boolean permutations, the same proofs work for all permutations.
\end{proof}

It remains to show the other direction of the inequality, which we will do in Proposition~\ref{prop:upper bound on runs}. To that end, we will define a function $\rho$, which maps a permutation $w$ to a shorter permutation (``shorter'' in terms of Coxeter length). It multiplies $w$ by a single run \textbf{on the left or right}, in such a way that a longest increasing subsequence in $\rho(w)$ is longer than that of $w$.

\begin{definition}\label{defn:the map rho}
We define a map 
$$\rho : \left(S_n \setminus \{12\cdots n\}\right) \rightarrow S_n$$
as follows. Fix a permutation $w \in S_n$ that is not the identity permutation, and consider the lexicographically least longest increasing subsequence in the one-line notation for $w$. (In fact, any longest increasing subsequence would satisfy our needs.) 
Let $q \in [1,n]$ be the smallest value not appearing in this subsequence. 
\begin{itemize}
\item If $q = 1$, then set $t := w^{-1}(1)$. Note that, by definition of $q$, we must have $t > 1$. Define $r$ to be the run $(t-1) \cdots 321$ and set $\rho(w) := w\rw{r}$. In other words, $\rho(w)$ uses the run to slide $1$ into the leftmost position of the permutation.
\item Suppose that $q>1$. If $q$ appears to the right of $q-1$ in the one-line notation for $w$, then set $t:= w^{-1}(q)$ and $t' := w^{-1}(q-1)$, so $t > t'$. 
In addition, $t > t'+1$ because otherwise adding $q$ to our lexicographically least longest increasing subsequence (which includes $q-1$, by definition of $q$) would have made a longer increasing subsequence. Define $r$ to be the (decreasing) run $(t-1) \cdots (t'+1)$ and set $\rho(w) := w\rw{r}$. That is, $\rho(w)$ uses the run to slide $q$ into position $t'+1$, the position immediately to the right of $q-1$.
\item If $q>1$ and $q$ appears to the left of $q-1$ in the one-line notation for $w$, then let $j \in [1,q-1]$ be the smallest value appearing to the right of $q$ in the one-line notation of $w$. 
Define $r$ to be the (increasing) run $j(j+1)\cdots(q-1)$ and set $\rho(w) := \rw{r}w$. In other words, there is a (possibly nonconsecutive) subsequence 
\newcolumntype{C}[1]{>{\centering\arraybackslash}p{#1}}
$$\begin{tabular}{C{.4in}C{.4in}C{.4in}C{.4in}C{.4in}C{.4in}C{.4in}}
$q$ & $j$ & $(j+1)$ & $(j+2)$ & $\cdots$ & $(q-2)$ & $(q-1)$
\end{tabular}$$
in $w$, and $\rho(w)$ transforms that subsequence into
$$\begin{tabular}{C{.4in}C{.4in}C{.4in}C{.4in}C{.4in}C{.4in}C{.4in}}
$j$ & $(j+1)$ & $(j+2)$ & $\cdots$ & $(q-2)$ & $(q-1)$ & \phantom{.}$q$.
\end{tabular}$$
\end{itemize}
\end{definition}

To make the process of Definition~\ref{defn:the map rho} more concrete, we consider a few examples. We will continue example (c) after a trio of lemmas, demonstrating those results as well.

\begin{example}\label{ex:rho examples}
\begin{enumerate}[(a)]
\item Let $u=342516$. The lexicographically least longest increasing subsequence in $u$ is $3456$. 
The smallest value not appearing in $3456$ is $q=1$, so we use the first scenario of Definition~\ref{defn:the map rho} to compute $\rho(u)$.
Set $t:=u^{-1}(q)=5$, producing the decreasing run $r=4321$. 
We set $\rho(u)=u\rw{r}$.
Multiplying $u$ by $\rw{r}$ on the right is equivalent to sliding $u(5)=1$ into the leftmost position of the one-line notation, so $\rho(u)=\textcolor{red}{1}34256$.

\item Let $v=142563$. The lexicographically least longest increasing subsequence in $v$ is $1256$. The smallest value not appearing in $1256$ is $q=3$. 
Since $3$ appears to the right of $3-1$, we use the second scenario of Definition~\ref{defn:the map rho} to compute $\rho(v)$. We set $t:=v^{-1}(q)=v^{-1}(3)=6$ and $t':=v^{-1}(q-1)=v^{-1}(2)=3$, producing the decreasing run $r=54$. 
We set $\rho(v)=v\rw{r}$. Multiplying $v=142563$ by $\rw{r}$ on the right will slide $v(t)=v(6)=q=3$ into position $t'+1=4$, the position immediately to the right of $q-1=2$, so $\rho(v)=142\textcolor{red}{3}56$.

\item Consider the permutation $w=51642738$. The lexicographically least longest increasing subsequence in $w$ is $1238$. The smallest value not appearing in this longest subsequence is $q=4$. 
Since $4$ appears to the left of $4-1$ in $w$, we use the third scenario of Definition~\ref{defn:the map rho} to compute $\rho(w)$, producing $j=2$ and the increasing run $r=23$. 
Then multiplying $w$ by $\rw{r}$ on the left transforms the subsequence 
$423$ of $w$ into $234$:
$\rho(w)=\rw{r}w=
516237\textcolor{red}{4}8$. 
\end{enumerate}
\end{example}

There are important features of the image of a permutation under the map $\rho$, and we will want to take advantage of these later. To prepare for that, we now identify three of the map's key qualities.

\begin{lemma}\label{lem:rho shortens a permutation}
For any nonidentity permutation $w$,
$$\ell(\rho(w)) + \ell(\rw{r})= \ell(w),$$
where $r$ is the run described in Definition~\ref{defn:the map rho}.
In particular, $\ell(\rho(w)) < \ell(w)$.
\end{lemma}

\begin{proof}
The minimality of $q$ and the definition(s) of the run $r$ in Definition~\ref{defn:the map rho} ensure that each simple reflection described by the letters in $r$ undoes an inversion of the permutation.
\end{proof}

While an application of the map $\rho$ decreases the length statistic, it increases a different permutation statistic.

\begin{lemma}\label{lem:rho has a longer subsequence}
For any nonidentity permutation $w$, the length of a longest increasing subsequence in $\rho(w)$ is greater than the length of a longest increasing subsequence in $w$.
\end{lemma}

\begin{proof}
The run $r$ in Definition~\ref{defn:the map rho} was constructed so that the effect of multiplying $w$ by $\rw{r}$ is to insert $q$ into the lexicographically least longest increasing subsequence in the permutation, without changing the relative order of any other letters in that subsequence.
\end{proof}

After establishing Theorem~\ref{thm:run and lambda1}, we shall see $\rho$ increases the length of a longest increasing subsequence by $1$.

Recall that Schensted's theorem relates those longest increasing subsequences to the size of the top row in the permutation's shape under the Robinson--Schensted--Knuth correspondence. 
This means that Lemma~\ref{lem:rho has a longer subsequence} can be written as
$$\lambda_1(\rho(w)) > \lambda_1(w).$$

Finally, because $\rho(w)$ and $w$ differ by product with a run, we can say something about the relationship between $\run(w)$ and $\run(\rho(w))$.

\begin{lemma}\label{lem:rho and runs}
For any nonidentity permutation $w$,
$$\run(w) \le \run(\rho(w)) + 1.$$
\end{lemma}

\begin{proof}
Since $\rho(w)$ is equal to $w\rw{r}$ or $\rw{r}w$ by Definition~\ref{defn:the map rho}, 
we can write $w$ as the product of an optimal run word for $\rho(w)$ and the inverse permutation $(\rw{r})^{-1}$. 
If we write the run $r$ as $r = r_1 \cdots r_h$, then $r_h \cdots r_1$ is also a run and it is a (in fact, the) reduced word for this $(\rw{r})^{-1}$. 
Thus it is possible to write $w$ as a product of $\run(\rho(w)) + 1$ runs.
This guarantees that an optimal run word for $w$ has at most $\run(\rho(w)) + 1$ runs, and so
$\run(w) \le \run(\rho(w)) + 1$.
\end{proof}

We demonstrate the preceding three lemmas using the third permutation from Example~\ref{ex:rho examples}.

\begin{example}\label{ex:lemmas about rho}
Consider the permutation $w=51642738$. As computed above,
$\rho(w)= 51623748$. 
\begin{enumerate}[(a)]
\item The $21$-patterns $42$ and $43$ appear in $w$ but not in $\rho(w)$. 
All other inversions of $w$ are also inversions of $\rho(w)$, so $\ell(\rho(w))=\ell(w)-\ell(\rw{r})=10-2=8$, illustrating Lemma~\ref{lem:rho shortens a permutation}.

\item In the construction of $\rho(w)$, the value $q=4$ was inserted into the increasing subsequence $1238$ to form $123\textcolor{red}{4}8$. This is an increasing subsequence of $\rho(w)$ that is longer than any increasing subsequence found in $w$, illustrating Lemma~\ref{lem:rho has a longer subsequence}.

\item We can compute $\run(w) = 4$. For example, three of the optimal run words for $w$ are:
\begin{center}
$\rw{32 \cdot 456 \cdot 321 \cdot 43}$\ , \
$\rw{32 \cdot 4321 \cdot 56 \cdot 43}$\ , \ and \ 
$\rw{32 \cdot 4321 \cdot 543 \cdot 6}$.
\end{center}
We highlight these particular optimal run words because their leftmost letters are ``$32$,'' which will be canceled after multiplication by $\rw{r}$. That is, we find three corresponding reduced words of $\rho(w)$:
\begin{center}
$\rw{456 \cdot 321 \cdot 43}$\ , \ 
$\rw{4321 \cdot 56 \cdot 43}$\ , \ and \ 
$\rw{4321 \cdot 543 \cdot 6}$.
\end{center} 
Therefore, $\run(\rho(w)) \le 3 = \run(w) - 1$, illustrating Lemma~\ref{lem:rho and runs}.
\end{enumerate}
\end{example}

Lemmas~\ref{lem:rho shortens a permutation}, \ref{lem:rho has a longer subsequence}, and~\ref{lem:rho and runs} give us tools for making inductive arguments involving the permutation statistics length, length of a longest increasing subsequence, and runs. In other words, these lemmas allow us to make inductive arguments toward achieving an upper bound for the function $\run(w)$. Such a bound, in turn, could be combined with the lower bound proved in Proposition~\ref{prop:lower bound on runs}.

\begin{proposition}\label{prop:upper bound on runs}
For any permutation $w \in S_n$, we have $\run(w) \le n - \lambda_1(w)$.
\end{proposition}

\begin{proof}
If $w$ is the identity, then $\lambda_1=n$ and $\run(w)=0$,  and the result is trivially true. Suppose that $w$ is not the identity permutation. Assume, inductively, that the result holds for all permutations shorter than $w$. In particular, since Lemma~\ref{lem:rho shortens a permutation} tells us that $\rho(w)$ is shorter than $w$, we have
$$\run(\rho(w)) \le n - \lambda_1(\rho(w)).$$
Next, note that, because $\lambda_1$ takes only integer values, we can rewrite Lemma~\ref{lem:rho has a longer subsequence} as
$$\lambda_1(w) \le \lambda_1(\rho(w)) - 1.$$
Combining this with Lemma~\ref{lem:rho and runs} gives
\begin{align*}
\run(w) &\le \run(\rho(w))+1\\
&\le n - \lambda_1(\rho(w)) + 1\\
&= n - (\lambda_1(\rho(w)) - 1)\\
&\le n - \lambda_1(w).
\end{align*}
\end{proof}

We can now prove the precise relationship between $\lambda_1$ and $\run$, for any permutation.

\begin{theorem}\label{thm:run and lambda1}
For any permutation $w \in S_n$,
$$\lambda_1(w) + \run(w) = n.$$
\end{theorem}

\begin{proof}
This follows from Propositions~\ref{prop:lower bound on runs} and~\ref{prop:upper bound on runs}.
\end{proof}

Intuitively speaking, multiplying a permutation on the right by a run is the same as deleting an entry from the permutation and then inserting it somewhere else. 
For example, let $w = 253146 \in S_6$.  
If we multiply $w$ on the right by the run 
$\rw{321}$, 
we get $w  
\rw{321} 
= 125346$. 
We can see that $1$ is deleted from $w$ and then inserted to the beginning. 
In fact, $\run(w)$ is the minimum number of runs whose concatenation is a decomposition for $w$, where we drop the reduced condition.

\begin{lemma}
The minimum number of runs whose concatenation is a decomposition for $w$ is $\run(w)$.
\end{lemma}

\begin{proof}
Clearly, the minimum number of runs whose concatenation is a decomposition for $w$ is at most $\run(w)$. We prove the other inequality.
Let $w\in S_n$ and let $w = \rho_1 \rho_2 \dots \rho_\ell$ where $\rho_i$ is a run for $1\leqslant i \leqslant \ell$. 
Since a run deletes one entry and inserts it somewhere else, it can change the length of a longest increasing subsequence by at most one. 
That is, the length of a longest increasing subsequence of the permutation $\rho_1 \rho_2 \dots \rho_\ell$ is at least $n-\ell$. 
So $\lambda_1(w)$ is at least $n-\ell$. It follows that $\ell \geqslant n-\lambda_1(w) = \run(w)$.
\end{proof}

With the equivalent definition of $\run(w)$ as being the minimum number of runs whose concatenation is a decomposition for $w$, our statistic $\run(w)$ recovers the {``Ulam distance"} on permutations. 
Ulam's metric was originally defined to study mutation of DNA sequences 
from the perspective of permutations
~\cite{Ula72} as well as to find the fastest way to sort a bridge hand of thirteen cards~\cite{AD99}. 
An \emph{Ulam move} 
deletes a value from the current permutation and places it at some other position. Correspondingly, the \emph{Ulam distance}  $U(\sigma,\tau)$ is the minimum number of Ulam moves needed to obtain $\tau$ from $\sigma$. 
See also~\cite{BB15}.
In fact, our Theorem~\ref{thm:run and lambda1} is equivalent to a statement about Ulam's distance given in ~\cite[Chapter~6B, Lemma~2]{Dia88}.  
Note that our proof is constructive (via Definition~\ref{defn:the map rho} of the map $\rho$), in contrast to the proof of the latter.

Note that, in general, applying the function $\rho$ is not the same as applying an Ulam move, as  
in the last case of $\rho$. 
We give an algorithm for sorting a permutation $w$ to the identity permutation by applying a shortest sequence of Ulam moves, as follows.
First, we apply $\rho$ repeatedly until we arrive at the identity permutation, giving us an optimal run word. Next, we read the runs of this optimal run word from right to left, one run at a time; each run corresponds to applying an Ulam move. 
In the following example, we demonstrate our algorithm for sorting a permutation.

\begin{example}\label{ex:Ulam 51642738 third case}
Continuing with Example~\ref{ex:lemmas about rho}, let $w = 51642738$. 
We apply $\rho$ repeatedly:
$\rho(w)=\rw{23}w=
51623{7}{4}8$
,  
$\rho^2(w)=\rw{1234}\rho(w)=12{6}347{5}8$,
$\rho^3(w)=\rw{345}\rho^2(w)=12345768$, and 
$\rho^4(w)=\rw{6}\rho^3(w)=12345678$. This gives an optimal run word $\rw{32 \cdot 4321 \cdot 543 \cdot 6}$ of $w$ which tells us how to sort $w$ using 
minimally many Ulam moves. Reading the four runs of this optimal run word from right to left, we apply the following Ulam moves to 51642738: 
\begin{enumerate}[1)]
    \item delete the sixth number, $7$, from $w$ and insert it after the number $3$, producing $51642378$; 
    \item delete the third number, $6$, and insert it after the number $3$, producing $51423678$; 
    \item delete the first number, $5$, and insert it after the number $3$, producing $14235678$; 
    \item delete the second number, $4$, and insert it after the number $3$, producing the identity permutation $12345678$. 
\end{enumerate}
\end{example}

\section{Canonical reduced words and the second row of RSK tableaux}\label{sec.second.row}

In this section, we construct and study a particular optimal run word for a boolean permutation $w$. We call it the \emph{RSK canonical reduced word} (\emph{canonical word} for short) of $w$ and denote it by $\canonicalword{w}$. In Subsection~\ref{subsection:canonical word}, we will present two different algorithms to produce $\canonicalword{w}$: in Definition~\ref{defn:canonical word}, we start from an arbitrary reduced word of $w$ and apply commutation relations; in Definition~\ref{defn:canonical word:heap}, we construct $\canonicalword{w}$ from the heap $H_w$, which gives a convenient visualization of $\canonicalword{w}$. In Subsection~\ref{subsection:canonical to rsk tableaux}, we establish an application of $\canonicalword{w}$. We show that $\canonicalword{w}$ directly produces $\P(w)$ and $\Q(w)$ without using the RSK insertion procedure.

\subsection{Constructing canonical words}\label{subsection:canonical word}
The first algorithm was inspired by a technique used in the proof of \cite[Theorem 6.4]{MT22}. Essentially, 
it uses commutation moves to push decreasing runs to the left within the word and increasing runs to the right, starting with runs on the smallest numbers. 

\begin{definition}
\label{defn:canonical word}
Let $w$ be a boolean permutation, and $\rw{s}\in R(w)$ an arbitrary reduced word.

\begin{enumerate}[Step (1):]
    \item 
\label{itm:defn:canonical word:let a be smallest}
Let $a$ be the smallest value appearing in $\rw{s}$. Apply commutation moves to push a run to the left or right of $w$ according to the following instructions.
\begin{enumerate}
\item  \emph{(Push a singleton run to the \textcolor{red}{left}.)}
If $a+1$ \textcolor{red}{does not appear} in $\rw{s}$, then define $w'$ so that $w=\textcolor{red}{\rw{a}}w'$.

\item \emph{(Push a decreasing run to the \textcolor{red}{left}.)}
If $a+1$ appears to the \textcolor{red}{left} of $a$ in $\rw{s}$, then let $b \ge a+1$ be the largest such that the run $b(b-1)\cdots a$ is a subsequence of $s$. 
Define $w'$ so that $w = \textcolor{red}{\rw{b (b-1)\cdots a}} w'$. 
\item \emph{(Push an increasing run to the \textcolor{red}{right}.)}
If $a+1$ appears to the \textcolor{red}{right} of $a$ in $\rw{s}$, 
then let 
$b \ge a+1$ be the largest 
such that the run $a\cdots(b-1)b$ is a subsequence of $s$. Define $w'$ so that $w = w' \textcolor{red}{\rw{a\cdots (b-1)b}}$.
\end{enumerate}

\item If $w'$ is not the identity permutation, repeat Step \eqref{itm:defn:canonical word:let a be smallest} on an arbitrary reduced word for $w'$. 
If $w'$ is the identity, 
we are done.
\end{enumerate}
The reduced word of $w$ created by this algorithm is $\canonicalword{w}$.
\end{definition}

Alternatively, we can construct the same canonical reduced word given in Definition~\ref{defn:canonical word} by using the heap  (defined in Section~\ref{sec:heap of boolean permutation})
of a boolean permutation. 

\begin{definition}
\label{defn:canonical word:heap}
Let $H$ be the heap diagram of a boolean permutation $w$, drawn in increasing order from left to right.
Start with two empty lists $\Decr:=\Decr{(H)}$ 
and $\Incr:=\Incr{(H)}$. We will scan the elements 
of $H$ from left to right and fill these two lists with decreasing and increasing runs, respectively.

\begin{enumerate}[Step (1)]
\item \label{itm:defn:canonical word:heap:let a be smallest}
Let $a$ be the leftmost element (that is, smallest number) of $H$.

\begin{enumerate}
\item 
If $a+1$ \textcolor{red}{is not an element} in $H$, then \textcolor{red}{append} the singleton run $\textcolor{red}{a}$ to the list of $\Decr$. 

\item 
If 
\textcolor{red}{
$a \succ a+1$} in $H$,
then let $b$ be the first extremal element of $H$ (necessarily minimal) to the right of $a$.
\textcolor{red}{Append} the decreasing run $\textcolor{red}{b(b-1)\cdots a}$ to the list $\Decr$.

\item 
If 
\textcolor{red}{
$a \prec a+1$} in $H$,
then let $b$ be the first extremal element of $H$  (necessarily maximal) to the right of $a$.
\textcolor{red}{Prepend} the increasing run $\textcolor{red}{a(a+1)\cdots b}$ to the list $\Incr$.
\end{enumerate} 

\item 
Let $H'$ be the diagram obtained by removing the singleton $a$ from $H$ (Case~(a)) 
or 
the elements $a, a+1, \dots, b$ (Case~(b) or~(c)).  If $H'$ is not empty, redefine $H:=H'$, 
and repeat Step~\eqref{itm:defn:canonical word:heap:let a be smallest}. 
If $H'$ is empty, we are done.
\end{enumerate}

\noindent If Dec is nonempty, 
let $\canonicalword{\textup{Dec}}$ be the concatenation of the decreasing runs in Dec, with smaller indices appearing first; otherwise, let it be the empty word. 
If Inc is nonempty, let
$\canonicalword{\textup{Inc}}$ be the concatenation of the increasing runs in Inc, with larger indices appearing first; otherwise, let it be the empty word. The reduced word $\canonicalword{w}$ is the concatenation of $\canonicalword{\textup{Dec}}$ and $\canonicalword{\textup{Inc}}$, in that order.
\end{definition}

We can verify that Definitions~\ref{defn:canonical word} and~\ref{defn:canonical word:heap} are equivalent, as follows. By Proposition \ref{prop:set of labeled linear extensions is commutativity class}, we know that a reduced word $\rw{s}$ for a boolean permutation $w$ corresponds to a linear extension of the heap $H_w$. In addition, we observe:
\begin{enumerate}
\item $x\in H_w$ if and only if $x \in \rw{s}$;
\item $x \succ x+1$ in $H_w$ if and only if $x+1$ is to the left of $x$ in $\rw{s}$;
\item $x \prec x+1$ in $H_w$ if and only if $x+1$ is to the right of $x$ in $\rw{s}$.
\end{enumerate}
Therefore we obtain the same decreasing run (including singleton) or increasing run at each iteration. Furthermore, appending a decreasing run is equivalent to pushing a decreasing run to the left while prepending an increasing run is the same as pushing an increasing run to the right.

Next, we demonstrate this canonical word construction using two boolean permutations: one having full support and one not.

\begin{example}\label{ex:canonical word in S10}
Consider the boolean permutation $314627(10)589 \in S_{10}$.
First, we construct $\canonicalword{w}$ following~Definition~\ref{defn:canonical word}.
\begin{itemize}
\item We start with an arbitrarily chosen reduced word: $\rw{259136847} \in R(w)$.
\item First, $a=1$. Since $a+1=2$ is to the left of $a$, and $b=2$, we push the decreasing run $\rw{21}$ to the left and write $w = \rw{\textcolor{red}{21} \cdot5936847}$.
\item 
Now we look at $w'=\rw{5936847}$. In this case, $a=3$. Since $a+1=4$ is to the right of $a$, and $b=4$, we push the increasing run $\rw{34}$ to the right: $w'=\rw{59687\cdot \textcolor{red}{34}}$.
\item 
Now we look at $w''=\rw{59687}$. Here $a=5$, we have $a+1=6$ to the right of $a$, and $b=7$, so we push the increasing run $\rw{567}$ to the right: $w''=\rw{98\cdot \textcolor{red}{567}}$.
\item 
We are left with $\rw{98}$, which is a run, so we are done.
\end{itemize}
Our steps produce the following;
$$  \rw{259136847} 
   \rightsquigarrow 
    \rw{21\cdot5936847} 
    \rightsquigarrow 
    \rw{21\cdot 59687\cdot34}
    \rightsquigarrow
    \rw{21\cdot98\cdot567\cdot34}
    = \canonicalword{w}.$$
We could also have constructed $\canonicalword{w}$ using heaps, according to Definition~\ref{defn:canonical word:heap}. The heap $H$ of $w$ is shown in Figure~\ref{fig:fence_reduced_word259136847}. 
As we go from left to right along the elements of $H$, we create $\textup{Dec}=\{ \textcolor{black}{\rw{21}}, \textcolor{black}{\rw{98}} \}$ 
and $\textup{Inc}=\{ \textcolor{black}{\rw{567}}, \textcolor{black}{\rw{34}} \}$. 
Then $\canonicalword{\textup{Dec}}=\textcolor{black}{\rw{21\cdot98}}$ and $\canonicalword{\textup{Inc}}=\textcolor{black}{\rw{567\cdot34}}$.
Their concatenation, in that order, produces the reduced word $\canonicalword{w}$ given above. 
\end{example}

The boolean permutation in the next example does not have full support, so its heap diagram is disconnected.

\begin{example}\label{ex:disconnected}
Consider the boolean permutation $w=231548697(11)(10) \in S_{11}$.
We can again construct $\canonicalword{w}$ following~Definition~\ref{defn:canonical word}.
\begin{itemize}
\item We start with an arbitrarily chosen reduced word: $\rw{471(10)268} \in R(w)$.
\item First, $a=1$. Since $a+1=2$ is to the right of $a$, and $b=2$, we push the increasing run $\rw{12}$ to the right and write $w = \rw{47(10)68\cdot \textcolor{red}{12}}$.
\item Now we look at $w'=\rw{47(10)68}$. In this case, $a=4$. Since $a+1=5$ does not appear in $\rw{47(10)68}$, we write $\rw{47(10)68}= \rw{\textcolor{red}{4}\cdot 7(10)68}$.
\item Now we look at $w''=\rw{7(10)68}$.
Here $a=6$. Since $a+1=7$ is to the left of $a$, and $b = 7$, we push the decreasing run $\rw{76}$ to the left and write $\rw{7(10)68}=\rw{\textcolor{red}{76}\cdot (10)8}$.
\item Now we look at $w'''=\rw{(10)8}$. Here $a=8$ and we push this singleton to the left and write $\rw{(10)8}= \rw{\textcolor{red}{8}\cdot (10)}$.
\item What remains is the singleton run $\rw{(10)}$, so we are done.
\end{itemize}
Our steps produce the following:
\begin{align*}
    \rw{471(10)268}
    &\rightsquigarrow
    \rw{47(10)68\cdot12} 
    \rightsquigarrow
    \rw{4\cdot7(10)68\cdot12}
    \rightsquigarrow
    \rw{4\cdot 76\cdot(10)8\cdot12}\\
    &\rightsquigarrow
    \rw{4\cdot 76\cdot8\cdot (10)\cdot12}
    =\canonicalword{w}.
\end{align*}
The heap construction would have produced the same result: the heap of $w$ is shown in Figure~\ref{fig:rw471_10_268}, creating $\textup{Dec}=\{ \textcolor{black}{\rw{4}}, \textcolor{black}{\rw{76}}, \textcolor{black}{\rw{8}}, \textcolor{black}{\rw{(10)}} \}$ and $\textup{Inc}=\{ \textcolor{black}{\rw{12}} \}$. 
Then $\canonicalword{\textup{Dec}}= \textcolor{black}{\rw{4\cdot76\cdot8\cdot(10)}}$ and $\canonicalword{\textup{Inc}}= \textcolor{black}{\rw{12}}$, and their concatenation, in that order, produces $\canonicalword{w}$.
\end{example}

\def\IncrColor{black} \def\DecrColor{black}
\begin{figure}[htb!]
\centering
\begin{tikzpicture}[xscale=1.5,yscale=1.5,>=latex]
\node(1) at (1,1) {$1$};
\node(2) at (2,0) {$2$}; 
\node(3) at (3,1) {$3$}; 
\node(4) at (4,2) {$4$};  
\node(5) at (5,1) {$5$};  
\node(6) at (6,2) {$6$}; 
\node(7) at (7,3) {$7$}; 
\node(8) at (8,2) {$8$}; 
\node(9) at (9,1) {$9$}; 

\draw[-] (1) -- (2); 
\draw[-] (2) -- (3); 
\draw[-] (3) -- (4); 
\draw[-] (4) -- (5); 
\draw[-] (5) -- (6); 
\draw[-] (6) -- (7);
\draw[-] (7) -- (8);
\draw[-] (8) -- (9);

\draw[rotate around={45:(1.5,0.5)}, dashed, \DecrColor] (1.5,0.5) ellipse (3mm and 11mm);

\draw[rotate around={-45:(3.5,1.5)}, dashed, \IncrColor] (3.5,1.5) ellipse (3mm and 11mm);

\draw[rotate around={-45:(6)}, dashed, \IncrColor] (6) ellipse (3.7mm and 20mm);

\draw[rotate around={45:(8.5,1.5)}, dashed, \DecrColor] (8.5,1.5) ellipse (3mm and 11mm);

\end{tikzpicture}
\caption{Heap diagram for the boolean permutation having canonical reduced word $\rw{21\cdot98\cdot567\cdot34}$.}
\label{fig:fence_reduced_word259136847}
    
\begin{tikzpicture}[xscale=1.5,yscale=1.5,>=latex]
\node(1) at (1,0) {$1$}; 
\node(2) at (2,1) {$2$}; 
\node(4) at (4,0) {$4$};  
\node(6) at (6,1) {$6$}; 
\node(7) at (7,0) {$7$}; 
\node(8) at (8,1) {$8$}; 
\node(10) at (10,0) {$10$}; 

\draw[-] (1) -- (2); 
\draw[-] (6) -- (7);
\draw[-] (7) -- (8); 

\draw[rotate around={-45:(1.5,0.5)}, dashed, \IncrColor] (1.5,0.5) ellipse (3mm and 11mm);

\draw[dashed, \DecrColor] (4) ellipse (3mm and 3mm);

\draw[rotate around={45:(6.5,0.5)}, dashed, \DecrColor] (6.5,0.5) ellipse (3mm and 11mm);

\draw[dashed, \DecrColor] (8) ellipse (3mm and 3mm);
\draw[dashed, \DecrColor] (10) ellipse (3mm and 3mm);
\end{tikzpicture}
\caption{Heap diagram for the boolean permutation having canonical reduced word $\rw{4\cdot76\cdot8\cdot(10)\cdot12}$.}
\label{fig:rw471_10_268}
\end{figure}

\subsection{From canonical reduced words to RSK tableaux}\label{subsection:canonical to rsk tableaux}

We now use $\canonicalword{w}$ to simply and directly construct $\P(w)$ and $\Q(w)$. 
Note that because a boolean permutation has at most two rows in its RSK partition, the RSK tableaux are completely determined by the values appearing in their second rows.

Let $\Rowone(T)$ (resp., $\RowTwo(T)$)  
 denote the contents of the first (resp., second) row of a tableau $T$. So, in particular, $\Rowtwo{w}$ and 
$\RowtwoQ{w}$ denote the second rows of $\P(w)$ and $\Q(w)$, respectively.

\begin{theorem}\label{thm:boolean second row}
If $w$ is boolean, then 
\begin{center}
$\Rowtwo{w}=$ $\{ i + 1$ $\mid$ $i$ is the leftmost entry in a run of $\canonicalword{w}\}$
\end{center}
and
\begin{center}
$\RowtwoQ{w}=$ $\{ i + 1$ $\mid$ $i$ is the rightmost entry in a run of $\canonicalword{w}\}.$
\end{center}
\end{theorem}

\begin{proof}
We first note that the two statements are equivalent, thanks to Proposition~\ref{prop:P w inverse is Q w} and the fact that reduced words for the inverse permutation $w^{-1}$ are exactly the reverse of the reduced words for $w$, from which it follows that $\canonicalword{w^{-1}}$ is the reverse word of  $\canonicalword{w}$ up to reordering the singleton runs. Therefore it suffices to prove the first statement, about $\Rowtwo{w}$.

It is straightforward to verify this result when $w$ is the identity permutation and when the canonical reduced word for $w$ consists of a single run. Assume now that the statement is true for all boolean permutations with canonical words having $k$ runs for some $k\geq 1$. Consider a boolean permutation $w\in S_n$ having $k+1$ runs in its canonical word.

Let $a$ be the smallest value in the support of $w$.
There are three cases:
\begin{enumerate}[I.]
\item \label{itm:thm:boolean second row:a}
$a+1$ does not appear in any reduced word for $w$,
\item \label {itm:thm:boolean second row:b}
$a+1$ appears to the left of $a$ in all reduced words for $w$, or
\item \label {itm:thm:boolean second row:c}
$a+1$ appears to the right of $a$ in all reduced words for $w$.
\end{enumerate}

\noindent \textbf{Case ~\ref{itm:thm:boolean second row:a}:} If $a+1$ does not appear in any reduced word for $w$, then $a$ is its own maximal run in all reduced words of $w$. Thus $w = \rw{a}w' = w'\rw{a}$ where $w'$ is a boolean permutation, $\canonicalword{w}=[a]\canonicalword{w'}$, $\supp(w') \subseteq \{a+2,\ldots, n-1\}$, and $\run(w') =  \run(w)-1$. Therefore $\Rowtwo{w'}\subseteq \{a+3,\dots,n\}$.
Furthermore, we have that $w'(i)=i$ for $i\leqslant a+1$, that is, the one-line notation of $w'$ is as follows:
\[
w' = 1\ \ 2 \ \ \cdots \ \ a \ \ (a+1) \ \ w'(a+2) \ \ \cdots \ \ w'(n)\,.
\]
Since $a+1 \notin \Rowtwo{w'}$ and the one-line notation for $w$ is
\[
w = 1\ \ 2 \ \ \cdots \ \ (a+1)\ \ a \ \ w'(a+2)\ \ \cdots \ \ w'(n)\,,
\]
we have that $\Rowtwo{w} = \Rowtwo{w'}\cup \{a+1\}$  and thus the inductive hypothesis on $w'$ completes the argument.

\smallskip

\noindent \textbf{Case~ 
\ref{itm:thm:boolean second row:b}:} 
Suppose $a+1$ is to the left of $a$ in all reduced words for $w$.
As in Definition~\ref{defn:canonical word}, let $b \ge a+1$ be the maximum value such that the decreasing run $\rw{b(b-1)\cdots a}$ is a subsequence of $\canonicalword{w}$.
Then $w = \rw{b \cdots (a+1)a} w'$, where $w'$ is a boolean permutation, $\canonicalword{w}=\rw{b \cdots (a+1)a} \canonicalword{w'}$, $\supp(w') \subseteq \{b+1,\ldots, n-1\}$, and $\run(w')= \run(w)-1$.

Because $1,\dots,b$ are not in the support of $w'$, these values are fixed by $w'$. The permutation $w$ is the result of multiplying $w'$ on the left by the run
$\rw{b(b-1)\cdots a}$. Therefore
$$w = 1 2 \cdots (a-1)\, \textcolor{red}{(b+1)a(a+1)\cdots (b-2)(b-1)}\,w(b+1)\cdots w(n)$$
where, for $b+1\leq i\leq n$, 
\[
w(i) = 
\begin{cases}
b & \text{ if $w'(i) =b+1$, and}\\
w'(i) & \text{ otherwise}.
\end{cases}
\]

Since the first $b$ values in the one-line notation of $w'$ are fixed, 
the partial insertion tableau $\P_b(w')$ is the 1-row insertion tableau of the identity permutation: $\P(12\cdots b)$. All entries to the right of $b$ in the one-line notation of $w'$ are larger than $b$, so none of the numbers $1,2,\dots,b$ will get bumped to the second row as we continue the insertion algorithm to produce $\P(w')$. Thus all numbers in $\Rowtwo{w'}$ are larger than $b$.

Meanwhile, the partial insertion tableau $\P_{b}(w)$ has $1, \ldots, (b-1)$ in the first row and $b+1$ in the second row: 
\[\P_b(w)=
\begin{ytableau}
1 & 2 & \null & \null & {\cdots} & \null & \null & {\scriptstyle b-1} \\
{\scriptstyle b+1}
\end{ytableau}. \]
All entries to the right of $b-1$ in the one-line notation of $w$ are larger than $b-1$, so none of the numbers $1,2,\dots, b-1$ will get bumped to the second row in the formation of $\P(w)$.

If $b+1$ bumps some larger value $w'(j)$ (necessarily $w'(j) \neq b+1$) in the construction of $\P(w')$, then the same value $w(j)=w'(j)$ would be bumped by $b$ in the construction of $\P(w)$. It follows that $\Rowtwo{w}= \Rowtwo{w'}\cup\{b+1\}$, and the inductive hypothesis on $w'$ completes the proof.

\smallskip

\noindent \textbf{Case~\ref {itm:thm:boolean second row:c}:} Suppose $a+1$ is to the right of $a$ in all reduced words for $w$.
As in Definition~\ref{defn:canonical word}, let 
$b \ge a+1$ be the maximum value in $\canonicalword{w}$ such that $\rw{a\cdots(b-1)b}$ is a subsequence of $\canonicalword{w}$. 
Then $w = w'[a(a+1) \cdots b]$ where $w'$ is a boolean permutation, $\canonicalword{w} = \canonicalword{w'}[a(a+1) \cdots b]$, $\supp(w') \subseteq \{b+1,\ldots, n-1\}$, and $\run(w')= \run(w)-1$.

The numbers $1,2,\dotsc,b$ are again fixed points of $w'$, so 
the one-line notation of $w'$ is of the form
$$w' = 1 2 \cdots (b-1)b\, w'(b+1)\cdots w'(n).$$
The fact that $w =w'\rw{a(a+1)\cdots b}$ means that
$$w = 1\cdots (a-1) \textcolor{red}{(a+1)(a+2) \cdots b \, w'(b+1)\, a\,} w'(b+2) \cdots w'(n).$$
Because $w'(b+1)>b$,
the first $b+1$ values in the one-line notation of $w'$ form an increasing sequence. Therefore
the partial insertion tableau $\P_{b+1}(w')$ is the $1$-row tableau:
\[
\ytableausetup
{mathmode, boxsize=11mm}
\P_{b+1}(w')=
\raisebox{-.15in}{\begin{ytableau}
1 & 2 & \null & \null & {\cdots} & \null & \null & b & {\scriptstyle w'(b+1)}
\end{ytableau}.} \]
Meanwhile, the partial insertion tableau $\P_{b+1}(w)$ has $1, \ldots, \widehat{a+1}, \dots, b,w'(b+1)$ in its first row and $a+1$ in the second row: 
\[
\ytableausetup
{mathmode, boxsize=11mm}
\P_{b+1}(w)=
\begin{ytableau}
1 & 2 & {\cdots} & a & {\scriptstyle a+2} & {\cdots} & b & {\scriptstyle w'(b+1)}\\
{\scriptstyle a+1}
\end{ytableau}. \]

The remaining steps of the RSK algorithm will bump exactly the same values for $w$ and for $w'$. Hence $\Rowtwo{w}= \Rowtwo{w'}\cup\{a+1\}$, and the result follows from the inductive hypothesis on $w'$.
\end{proof}

We demonstrate Theorem~\ref{thm:boolean second row} by recalling a previous example.

\begin{example}
Let $w = 314627(10)589$, and recall that $\canonicalword{w} = \rw{21\cdot 98\cdot 567\cdot 34}$, as computed in Example~\ref{ex:canonical word in S10}. 
The RSK tableaux for $w$ are 
\[\ytableausetup
{mathmode, boxsize=normal}
\P(w)=\begin{ytableau}1& 2& 5& 7& 8& 9\\ 
3& 4& 6& 10
\end{ytableau} \text{ \ and \ } 
\Q(w)=\begin{ytableau}1& 3& 4& 6& 7& 10\\ 2& 5& 8& 9
\end{ytableau},\]
confirming Theorem~\ref{thm:boolean second row}. That is, $\Rowtwo{w} = \{2+1, 9+1, 5+1, 3+1\}$ and $\RowtwoQ{w} = \{1+1, 8+1, 7+1, 4+1\}$.
\end{example}

\begin{corollary}\label{cor:canonicalword is optimal run word}
If $w$ is a boolean permutation, then $\canonicalword{w}$ is an optimal run word for $w$.
\end{corollary}

\begin{proof}
This follows immediately from Theorem~\ref{thm:run and lambda1} and Theorem~\ref{thm:boolean second row}.
\end{proof}

Another interesting consequence of this result is that certain values cannot appear together in the second row of $\P(w)$ when $w$ is boolean. We present an example of this here, and the result will be generalized in Section \ref{sec:Characterizing boolean insertion tableaux}.

\begin{corollary}\label{cor:three in a row can't be in row 2}
If $w$ is a boolean permutation, then $\{i,i+1,i+2\}\not\subseteq \Rowtwo{w}$ for all $i$.
\end{corollary}

\begin{proof}
Let $w$ be a boolean permutation with $i,i+1\in \Rowtwo{w}$. Thus the canonical reduced word for $w$ has one run, call it $r_{i-1}$, with leftmost element $i-1$ and another run, call it $r_{i}$, with leftmost element $i$. 
In particular, $r_{i}$ is either a singleton or an increasing run.

Recall the algorithm used to construct $\canonicalword{w}$, given in Definition~\ref{defn:canonical word}. If $r_{i}$ is an increasing run, then $i+1$ must be part of $r_{i}$. On the other hand, if $r_{i}$ is a singleton, then $i+1$ does not appear in any reduced word for $w$. In either case, it follows from Theorem~\ref{thm:boolean second row} that $i+2\notin \Rowtwo{w}$.
 \end{proof}

\section{Characterizing boolean insertion tableaux}
\label{sec:Characterizing boolean insertion tableaux}

As we have observed, the insertion tableau of a boolean permutation has at most two rows. On the contrary, not every $2$-row standard tableau is the insertion tableau of some boolean permutation. For example, as a consequence of Corollary~\ref{cor:three in a row can't be in row 2}, the following tableau is not the insertion tableau of any boolean permutation.
\[\young(1235,4678)
\]
In this section, we will characterize the 2-row standard tableaux that are insertion tableaux for boolean permutations. We will rely heavily on the definition of the canonical reduced word from Section~\ref{subsection:canonical word}.

\begin{definition}\label{defn:balanced} 
Let $L$ be a set of integers. If, for all integers $x$ and $y$, with $x>0$, we have 
$$|[y,y+2x] \cap L| \le x+1,$$ 
then we will say that $L$ is \emph{uncrowded}.
Otherwise, we say that $L$ is \emph{crowded}.
\end{definition}

In other words, 
a set of integers $L$ is {crowded} if 
$L$ contains more than $x+1$ of the integers in some interval of $2x+1$ integers. 
We are interested in crowded and uncrowded sets as they pertain to standard tableaux. The following technical lemma is important for Theorem \ref{thm:run lead rules for uncrowded tableaux} and Definition \ref{def:uncrowded tableau}.

\begin{lemma}\label{lem:uncrowded second rows}
Let $T$ be a standard tableau with at most two rows, and let $R\subseteq \RowTwo{(T)}$. Then $R\cup\{1\}$ is uncrowded if and only if  $R$ is uncrowded. In particular, $\RowTwo{(T)}\cup \{1\}$ is uncrowded if and only if $\RowTwo{(T)}$ is uncrowded.
\end{lemma}

\begin{proof}
One direction of this statement is clear since every subset of an uncrowded set is uncrowded. To prove the other direction, assume for the sake of contradiction that $R$ is uncrowded but $R\cup\{1\}$ is crowded. This means there is some minimal $x>0$ such that $|[1,2x+1]\cap (R\cup\{1\})|>x+1$.

Therefore, there are at least $x+2$ elements in $R$ from the set
\[
\{1,2,\dots, 2x-1, 2x, 2x+1\}.
\]
Furthermore, since $x$ is minimal, there are at most $x$ elements in $R$ from the set
\[
\{1,2,\dots, 2x-1\}. 
\]  
Hence  
$\{2x,2x+1\}\subseteq R$ and
there are exactly $x$ elements in $R$ from the set $\{1,2,\dots, 2x-1\}$. 
Since $R$ is uncrowded, $\{1,2,\dots, 2x-1\} \cap R  = \{2,4,\dots,2x-2\}$.
It follows that $\{2,4,\ldots, 2x, 2x+1\}\subseteq R$.

This is a set of size $x+1$, and there are only $x$ positive integers smaller than $2x+1$. However, $\Rowone{(T)}$ requires at least $x+1$ such numbers, so this is a contradiction.  
\end{proof}

With this lemma in hand, we can prove the main result in this section which will relate uncrowded sets to boolean permutations.

\begin{theorem}\label{thm:run lead rules for balanced tableaux}
\label{thm:run lead rules for uncrowded tableaux}
Let $L$ be a subset of $\{1, \ldots, n-1\}$. Then $L \cup \{0\}$ is uncrowded if and only if $L$ is the set of leftmost letters in the runs of the canonical reduced word of a boolean permutation.
\end{theorem}

\begin{proof}
First note that the result is easily checked when $|L|$ is small.

Let us now prove the direction that if $L \cup \{0\}$ is a crowded set then  $L$ cannot be the set of leftmost run letters in the canonical reduced word of a boolean permutation.
Suppose that $L \cup \{0\}$ is crowded, and let us find a minimally wide set demonstrating this crowding: fix a pair of values $y\in L\cup \{0\}$  and $x>0$ such that (1) $|[y,y+2x] \cap (L \cup \{0\})| > x+1$ and (2) $x$ is minimal for this and all other possible values of $y$.

The second condition means no proper subset of $[y,y+2x] \cap (L \cup \{0\})$ is crowded. This implies that one of the following cases holds:
\begin{enumerate}[(i)]
\item \label{itm:proof:uncrowded set and boolean:i}
$\{1,3,5,\ldots, 2x-1,2x\} \subseteq L$,
\item \label{itm:proof:uncrowded set and boolean:ii}
$\{y,y+1,y+2\}\subseteq L$, or
\item \label{itm:proof:uncrowded set and boolean:iii}
$\{y,y+1,y+3,y+5,\ldots,y+2x-1,y+2x\}\subseteq L$.
\end{enumerate}
We will now prove in all cases that these cannot be leftmost run letters in the canonical word for a boolean permutation.
 
First, assume that Case~\eqref{itm:proof:uncrowded set and boolean:i} holds.  
Suppose by contradiction that $1,3,5,\ldots, 2x-1,2x$ are leftmost elements in the runs of the canonical reduced word $\canonicalword{b}$ of some boolean permutation $b$. 
Then Theorem~\ref{thm:boolean second row} tells us that  $\Rowtwo{b}\supseteq\{2,4,6,\dots,2x,2x+1\}$. However, $\{2,4,6,\dots,2x,2x+1\}$ is uncrowded and $\{1,2,4,6,\dots,2x,2x+1\}$ is crowded. 
By Lemma \ref{lem:uncrowded second rows}, we conclude this case is impossible.

Case~\eqref{itm:proof:uncrowded set and boolean:ii} is Corollary~\ref{cor:three in a row can't be in row 2}. 

Next, assume that  Case~\eqref{itm:proof:uncrowded set and boolean:iii} holds.
 Suppose the elements of the set $\{y,y+1,y+3,\dots,y+2x-1\}$ are leftmost letters in  the canonical reduced word $\canonicalword{b}$ of a boolean permutation $b$. 
 Then $\canonicalword{b}$ must include the run
 \[
 y(y-1)\cdots,
 \]
which forces those remaining leftmost letters to come from the runs:
\begin{align*}
&(y+1){\color{red}(y+2)}\\
&(y+3){\color{red}(y+4)}\\
    & \ \ \vdots\\
    & (y+2x-1){\color{red}(y+2x)},
\end{align*}
where the rightmost values are either in those runs or are not in the support of $b$. 
In either case, it would be impossible for $\textcolor{red}{y+2x}$ to be a leftmost run letter in $\canonicalword{b}$. 

Thus, if $L \cup \{0\}$ is crowded, then $L$ cannot be the set of leftmost run letters of the canonical reduced word of a boolean permutation.

We now show that, whenever $L\cup \{0\}$ is uncrowded, there is a canonical word for some boolean permutation whose leftmost run letters are exactly the elements of $L$. 
For $L=\emptyset$, the empty word for the identity permutation satisfies these conditions. Now say $L\cup \{0\}$ is uncrowded, and assume, inductively, that our result holds for all $L'$ with $|L'|<|L|$. Write $L = \{m_0<\cdots < m_k\}$, and note that, for all $i$, we have $m_i\geq 2i+1$. Indeed, if there is an $i>0$  such that $m_i\leq 2i$, then $L\cup \{0\}$ is crowded since $$|[0,2i]\cap (L\cup \{0\})| =|\{0, m_0, \ldots, m_i\}|=i+2>i+1.$$

First, consider the case that $m_i = 2i+1$ for all $i$. Then $L = \{1,3,\ldots, 2k+1\}$, and the following word -- comprised of  two-element increasing runs -- is the canonical word for a boolean permutation with leftmost run letters exactly the elements of $L$: $[(2k+1)(2k+2)\cdot \cdots \cdot 34\cdot 12].$

Otherwise, choose the smallest $j$ such that $m_j>2j+1$, and define $$L' \coloneqq \{z-m_j: z\in L \textup{ and } z>m_j\}.$$ Note that $L'\cup \{0\} \subseteq \{z-m_j:z\in L\cup \{0\}\}$. The latter set is uncrowded since it is a shift of the uncrowded set $L\cup \{0\}$. This means $L'\cup \{0\}$ is a subset of an uncrowded set and is therefore also uncrowded. 
By the inductive hypothesis, there exists a canonical word $[s]$ for a boolean permutation with leftmost run letters exactly the elements of $L'$.

Define $[t]$ to be the word obtained by adding $m_j$ to each letter in $[s]$. Since $[t]$ is simply a shift of $[s]$, we see that $[t]$ is a canonical word for a boolean permutation with leftmost run letters $L\setminus\{m_0,\ldots, m_j\}=L\setminus\{1,3,\ldots, 2j-1,m_j\}$. By construction, all letters of $[t]$ are larger than $m_j$. Since $m_j>2j+1$, we also have $m_j-1>2j$. Thus, we conclude that the following word is a canonical word for a boolean permutation with leftmost run letters given by $L$:
$$[m_j(m_j-1)\cdot t \cdot (2j-1)(2j)\cdot \cdots \cdot 34\cdot 12].$$
\end{proof}

Combining Theorem~\ref{thm:run lead rules for uncrowded tableaux} and Theorem~\ref{thm:boolean second row}, we can characterize the sets $\Rowtwo{w}$ when $w$ is boolean.

\begin{corollary}\label{cor:second row rules for balanced tableaux}
Let $X$ be a subset of $\{2, \ldots, n\}$, and set $L := \{x-1 : x \in X\}$. The set $X$ is equal to $\Rowtwo{w}$ for some boolean permutation $w \in S_n$ if and only if $L \cup \{0\}$ is uncrowded.
\end{corollary}

This result together with Lemma \ref{lem:uncrowded second rows} motivates us to define a special class of tableaux.

\begin{definition}\label{def:uncrowded tableau}
Let $T$ be a standard tableau having at most two rows. We say that $T$ is \emph{uncrowded} if the set $\RowTwo(T)$ is uncrowded.
\end{definition}

Note that a $1$-row tableau, for which $\RowTwo(T) = \emptyset$, is always uncrowded. We can rephrase Definition \ref{def:uncrowded tableau} to say a tableau is uncrowded if and only if it is the insertion tableau of a boolean permutation. Note that this does not mean that an uncrowded tableau can \emph{only} occur as an insertion tableau of a boolean permutation. Because the RSK insertion algorithm is a bijection between permutations and \emph{pairs} of standard tableaux, multiple permutations can have the same insertion tableau. For example, both $3142$ (boolean) and $3412$ (not boolean) have the following insertion tableau which is uncrowded because  $\{3,4\}$ is uncrowded.
\[\begin{ytableau}
1 & 2\\
3 & 4
\end{ytableau}\]

The inverse of a boolean permutation is also boolean, because reversing a reduced word will not introduce any repetition among its letters. Therefore, due to Proposition~\ref{prop:P w inverse is Q w}, there is a statement about recording tableaux that is analogous to the insertion tableau result given in Corollary~\ref{cor:second row rules for balanced tableaux}.

\begin{corollary}\label{cor:second row rule for recording tableaux}
Let $Q$ be a standard tableau with at most two rows. Then $Q$ is the RSK recording tableau of some boolean permutation if and only if $\RowTwo(Q)$ is uncrowded. 
\end{corollary}

In other words, both the insertion and recording tableaux of boolean permutations follow the same characterization: their second rows are uncrowded.
However the converse is not true. For example, $3412$ is not a boolean permutation but has uncrowded insertion and recording tableaux.

\section{Enumerating uncrowded tableaux}
\label{sec:counting uncrowded tableaux}

In this section, we enumerate uncrowded tableaux via a bijection to a certain set of binary words. A \emph{maximal run} in a word is a factor of maximally many identical symbols, and it is common to use \emph{run-length} to describe the length of a maximal run. 
For example, the binary word $0111001$ has run-lengths $1$, $3$, $2$, and $1$, when read from left to right. 

Let $U_n$ be the set of standard uncrowded tableaux of size $n$ 
(including the $1$-row tableau), and let $X_n$ be the set of $01$-words of length $n-1$ in which all run-lengths of $1$s are odd. 
We define a map $f:X_n \rightarrow U_n$ as follows. 
(See Example \ref{ex:illustrating the map f}.)
If $x$ is the word whose letters are all $0$s, then let \[f(x)=
\raisebox{-.05in}{
\begin{ytableau}
1&2&\cdots &n
\end{ytableau}}.\] 
Otherwise, given $x = x_1x_2\cdots x_{n-1} \in X_n$, we will construct a $2$-row tableau $f(x)$. Define $\alpha(x) \subset [1,n]$ 
and 
 $\beta(x) \subset [2,n]$ 
 as follows:
\begin{itemize}
\item $1 \in \alpha(x)$.
\item If $x_i = 0$ then $n+1-i \in \alpha(x)$.
\item If we have a maximal run of $1$s starting at index $i$ and ending at index $i+2k$, then
\begin{itemize}
\item[$\star$] $n+1-i \in \beta(x)$, 
\item[$\star$] if $2k>0$, then $n+1-j \in \beta(x)$ for all $j \in \{i+1,i+3,i+5,\ldots, i+2k-1\}$, and
\item[$\star$] if $2k>0$, then $n+1-j \in \alpha(x)$ for all $j \in \{i+2,i+4,i+6,\ldots, i+2k\}$.
\end{itemize}
\end{itemize}
Let $f(x)$ be the tableau whose first (resp., second) row is the increasing sequence of entries in $\alpha(x)$ (resp., $\beta(x)$).

We first verify $f$ is a well-defined map from $X_n$ to $U_n$.

\begin{lemma}
For each $x\in X_n$, we have $f(x) \in U_n$.
\end{lemma}

\begin{proof}
We must establish two facts: $\beta(x)$ is uncrowded, and $f(x)$ is a standard tableau.

The requirement that each run-length of $1$s is odd means that, for each even length interval of integers $[i,i+2z]$, we can have at most 
\[n+1-i \text{ and } \{n+1-j \mid j \in \{i+1,i+3,\dots, i+2z-1\} \},\quad \text{or}
\]
\[
\{n+1-j \mid j \in \{i,i+2,\dots, i+2z\} \}
\]
in $\beta(x)$. Therefore we can have at most $z+1$ integers in $\beta(x)$ from $[i,i+2z]$.
Therefore, $\beta(x)$ is uncrowded.

To see that $\alpha(x)$ and $\beta(x)$ form the first and second rows, respectively, of a standard tableau, we observe that the following is an injective map sending an entry in the second row to a smaller entry in the first row: 
For a maximally long factor $x_{i} \cdots x_{i+2k}$ of $1$s in $x$, consider the entries of the interval $[n+1-(i+2k-1),n+1-i]$ in $\beta(x)$.

\begin{itemize}
\item 
If $2k=0$, then the corresponding singleton set in $\beta$ is $\{ n+1-i \}$. 
Either $i=n-1$ or $x_{i+1}=0$. If $i=n-1$, then $n+1-i=n+1-(n-1)=2$ and we map $2\in \beta(x)$ to $1 \in \alpha(x)$. 
If $x_{i+1}=0$, then $n+1-(i+1)=n-i \in \alpha(x)$, and we map $n+1-i \in \beta(x)$ to $n-i \in \alpha(x)$.

\item 
If $2k \geq 2$, then we have
\[\{n+1-i \} \cup \{ n+1-j \mid j =i+1,i+3,\ldots, i+2k-1\} \subseteq \beta(x)\]
and
\[\{ n+1-j \mid j = i+2,i+4,\ldots, i+2k\} \subseteq \alpha(x).\]
Either $i+2k=n-1$ or $x_{i+2k+1}=0$. 
If $i+2k=n-1$, then we map $n+1-(2k-1)\in \beta(x)$ to $1\in \alpha(x)$. 
Otherwise, the fact that $x_{i+2k+1}=0$ gives us an entry in $\alpha(x)$ that is smaller than everything in $\{ n+1-i\} \cup \{ n+1-j \mid j =i+1,i+3,\ldots, i+2k-1\} \subseteq \beta(x)$. 
This allows us to send each element in $\beta(x)$ to a unique smaller entry in $\alpha(x)$.
\end{itemize}
\noindent Therefore each $f(x)$ is indeed in $U_n$.
\end{proof}

In fact, not only does $f$ send elements from $X_n$ to the set $U_n$, but we can invert this process. That is, the map $f$ is a bijection.

\begin{proposition}\label{prop:f is a bijection}
The map $f: X_n \to U_n$ is a bijection.
\end{proposition}

\begin{proof}
We describe the inverse map $g: U_n \to X_n$ of $f$. 

If $T\in U_n$ is the $1$-row tableau, then let $g(T) $ be the word whose letters are all $0$s. 

Otherwise, let $g(T)$ be the $01$-word $x_1x_2\cdots x_{n-1}$ constructed using the following algorithm. 
Set $x_i:=0$ for $1\le i \le n-1$.
Let $\beta$ denote the second row of $T$. 
Since $T$ is standard, we have $\beta\subset [2,n]$.  
\begin{itemize}
\item Let $z$ be the largest number in $\beta$. Note that $z>1$, and set $x_{n+1-z}:=1$.
\begin{itemize}
\item[$\star$] If $z-1\notin \beta$, then let $\beta':=\beta \setminus \{ z \}. $ 
\item[$\star$] 
Otherwise the uncrowded condition guarantees the existence of a maximal integer $k\geqslant 1$ such that 
$$[z-2k,z] \cap \beta = \{z,z-1,z-3,\dots,z-(2k-1)\}.$$
Set 
\[
\text{$x_j:=1$ for $n+2-z\le j \le n+(2k+1)-z$,}
\]
and let $\beta':=\beta \setminus \{z,z-1,z-3,\dots,z-(2k-1)\}$.
\end{itemize}
\item 
If $\beta'$ is empty, then we are done. Otherwise, redefine $\beta:= \beta'$ and iterate the process.
\end{itemize}

We first check that the algorithm 
is well-defined. More precisely, we need to prove that $n+(2k+1)-z$ is less than or equal to $n-1$, which is equivalent to proving that $z$ is greater than or equal to $2k+2$.
Suppose, for the sake of contradiction, that $z\le 2k+1$. Because $z-(2k-1) \in \beta$, we have $z-(2k-1)\geq 2$, and thus $z=2k+1$. 
The $k+1$ elements of $S:= \{z,z-1,z-3,\dots,z-(2k-1)\}$ are in $\RowTwo(T)$. When $z = 2k+1$, we have $S = \{2, 4, \ldots, 2k-1, 2k\}$. This means $S\subseteq\RowTwo(T)$ is uncrowded and $S\cup \{1\}$ is crowded. Lemma \ref{lem:uncrowded second rows} shows this is impossible, and we have reached a contradiction

Next we show that $x_{n+(2k+2)-z}$, if it exists, will stay equal to $0$ 
at the conclusion of each iteration.  To see this, note that  $z-(2k+1)$ is not in the second row of $T$, so the next largest element in the iteration will be less than $z-(2k+1)$.

Now it is apparent that the above algorithm indeed gives a $01$-word. 
Moreover, all run-lengths of $1$s are odd, therefore $g(T)\in X_n$. 
It is straightforward to check that $f\circ g = \mathbf{1}_{U_n}$ and $g\circ f = \mathbf{1}_{X_n}$. Therefore we obtain a bijection between $U_n$ and $X_n$.
\end{proof}

We demonstrate this bijection with an example, beginning with the map $f$.

\begin{example}\label{ex:illustrating the map f}
Suppose we have the $01$-word
\setcounter{MaxMatrixCols}{20}
\[
\begin{matrix}
& x_1 & x_2 & x_3
& x_4 & x_5 & x_6 
& x_7 & x_8 & x_9 
& x_{10} & x_{11} & x_{12} 
& x_{13} & x_{14} & x_{15} 
& x_{16} & x_{17} \\ 
x= & 1 & 0 & 0 & 1 & 0 & 1 & 0 & 1 & 1 & 1 & 1 & 1 & 0 & 1 & 1 & 1 & 0
& \in X_{18}.
\end{matrix}
\]
Then 
$\alpha(x)=\{1,2,3,6,7,9,12,14,16,17\}$, $\beta(x) 
= \{4,5,8,10,11,13,15,18\}$ and 
\begin{equation}
\label{eq:ex:T in U18}
T = f(x) = 
\raisebox{.05in}
{\begin{ytableau}
1 & 2 & 3 & 6 & 7 & 9 & 12 & 14 & 16 & 17 \\
4 & 5 & 8 & 10 & 11 & 13 & 15 & 18
\end{ytableau} }
\in U_{18}.
\end{equation}

For example, the letter $x_{13}=0$ tells us that $n+1-13=6 \in \alpha(x)$, while the maximal block $x_8 \cdots x_{12}$ of $1$s tells us that
\begin{itemize}
\item[$\star$] $n+1-j \in \beta(x)$ for all $j \in \{8,8+1,8+3\}$, so $\{11, 10, 8 \} \subseteq \beta(x)$
and
\item[$\star$] $n+1-j \in \alpha(x)$ for all $j \in \{8+2,8+4\}$, so $\{ 9, 7\} \subseteq \alpha(x)$. 

\end{itemize}
\end{example}

Next, we illustrate the inverse map $g$. 

\begin{example}
Suppose we have $T\in U_{18}$ as given in~\eqref{eq:ex:T in U18}, and let $\beta = \{18,15,13,11,10,8,5,4\}$ denote the second row of $T$.
\begin{itemize}
    \item[$\star$] $z = 18\in \beta$ and $z-1\notin \beta$. Thus $x_1 = 1$.
    \item[$\star$] $z = 15\in \beta$ and $z-1\notin \beta$. Thus $x_4=1$.
    \item[$\star$] $z = 13\in \beta$ and $z-1\notin \beta$. Thus $x_6=1$.
    \item[$\star$] $z = 11\in \beta$ and $z-1,z-3\in \beta$, thus $k = 2$. Therefore $x_8=x_9=x_{10}=x_{11}=x_{12}=1$.
    \item[$\star$] $z = 5\in \beta$ and $z-1\in \beta$, thus $k = 1$. Therefore $x_{14}=x_{15}=x_{16}=1$.
\end{itemize}
Therefore we get $g(T) = 10010101111101110$, and this is exactly the word $x$ from Example~\ref{ex:illustrating the map f}.
\end{example}

The bijection developed above, between uncrowded tableaux $U_n$ and the set $X_n$ of $01$-words having odd run-lengths of $1$s, allows us to enumerate the set $U_n$, and related subsets. We present those results in the following corollary. Note that we must adjust for indexing in parts \eqref{cor:itm:oeisA077865} and \eqref{cor:itm:oeisA006053} of the result, because both of the sequences referenced there are for binary words of length $n$, not $n-1$. 

\begin{corollary}
\begin{enumerate}[(a)]
\item \label{cor:itm:oeisA028495}
The number of uncrowded tableaux in $U_n$ are counted by the sequence~
\cite[A028495]{oeis}, which is known to count the $01$-words in $X_n$. Below is a table for the first few entries of $|U_n|$.
\begin{center}
\begin{tabular}
{c||p{.25in}|p{.25in}|p{.25in}|p{.25in}|p{.25in}|p{.25in}|p{.25in}|p{.25in}|p{.25in}|p{.25in}}
\raisebox{0in}[.2in][.1in]{}$n$ & 1 & 2 & 3 & 4 & 5 & 6 & 7 & 8 & 9 & 10\\
\hline
\raisebox{0in}[.2in][.1in]{}$|U_n|$ & 1 & 2 & 3 & 6 & 10 & 19 & 33 & 61 & 108 & 197
\end{tabular}
\end{center}

\item \label{cor:itm:oeisA077865}
The number of $2$-row uncrowded tableaux in $U_n$, which is $U_n-1$, are counted by the sequence~
\cite[A077865]{oeis}, which is known to count the $01$-words in $X_n$, not including the all-$0$s word.

\item \label{cor:itm:oeisA006053}
The number of uncrowded tableaux in $U_n$ having $n$ in the second row are counted by the sequence~
\cite[A006053]{oeis}, which is known to count the $01$-words in $X_n$ that start with the letter $1$.
Below is a table for the first few entries for the sequence.
\begin{center}
\begin{tabular}
{c||p{.25in}|p{.25in}|p{.25in}|p{.25in}|p{.25in}|p{.25in}|p{.25in}|p{.25in}|p{.25in}|p{.25in}}
\raisebox{0in}[.2in][.1in]{}$n$ & 1 & 2 & 3 & 4 & 5 & 6 & 7 & 8 & 9 & 10\\
\hline
\raisebox{0in}[.2in][.1in]{}$a(n)$ & 0 & 1 & 1 & 3 & 4 & 9 & 14 & 28 & 47 & 89
\end{tabular}
\end{center}
\end{enumerate}
\end{corollary}

\section*{Acknowledgements}
The authors would like to thank the 2021--2022 Research Community in Algebraic Combinatorics program at ICERM, through which this research took place. We thank the organizers and staff for putting together this invigorating and inspiring workshop series. 
The authors are also grateful to Carolina Benedetti for helpful discussions. 
This work also benefited from computation using {\sc SageMath}~\cite{sage}. Finally, we thank the anonymous reviewers whose suggestions helped improve and clarify this paper.

\subsection*{Data Availability} Data sharing is not applicable to this article as no datasets were generated or analyzed during the current study.

\subsection*{Conflict of Interest} The authors declare that they have no conflicts of interest.

\end{document}